\documentclass[12pt]{amsart}
\makeatletter
\newcommand*{\rom}[1]{\expandafter\@slowromancap\romannumeral #1@}

\usepackage{amsmath,amsthm,amssymb,comment,fullpage}

\usepackage[T1]{fontenc}
\usepackage{latexsym}
\usepackage[dvips]{graphics}
\usepackage{epsfig}
\usepackage{yhmath}
\usepackage{amsmath,amsfonts,amsthm,amssymb,amscd}
\input amssym.def
\input amssym.tex
\usepackage{color}
\usepackage{bbold}
\usepackage{hyperref}
\usepackage{amsthm}
\usepackage{url}
\usepackage{dsfont}

\numberwithin{equation}{section}

\newtheorem{thm}{Theorem}[section]

\theoremstyle{plain}

\newtheorem{corollary}[thm]{Corollary}
\newtheorem{definition}[thm]{Definition}

\newtheorem{lemma}[thm]{Lemma}

\newtheorem{theorem}[thm]{Theorem}

\newtheorem*{notation*}{Notation}

\newcommand\be{\begin{equation}}
\newcommand\ee{\end{equation}}
\newcommand\bea{\begin{eqnarray}}
\newcommand\eea{\end{eqnarray}}
\newcommand\bi{\begin{itemize}}
	\newcommand\ei{\end{itemize}}
\newcommand\ben{\begin{enumerate}[(a)]}
	\newcommand\een{\end{enumerate}}
\newcommand\bc{\begin{center}}
	\newcommand\ec{\end{center}}
\def\ba#1\ea{\begin{align*}#1\end{align*}}

\newcommand{\R}{\ensuremath{\mathbb{R}}}
\newcommand{\C}{\ensuremath{\mathbb{C}}}
\newcommand{\Z}{\ensuremath{\mathbb{Z}}}

\newcommand{\N}{\mathbb{N}}

\newcommand{\D}{\mathbb{D}}

\usepackage{enumerate}

\begin{document}
	\title{Uniqueness of a Furstenberg system}
	
	
	\author{Vitaly Bergelson}
	\address{Department of Mathematics, Ohio State University, Columbus, OH 43210, USA}
	\email{vitaly@math.ohio-state.edu}
	\thanks{}
	
	\author{Andreu Ferr\'e Moragues}
	\address{Department of Mathematics, Ohio State University, Columbus, OH 43210, USA}
	\email{ferremoragues.1@osu.edu}
	\maketitle
	\begin{abstract}
		Given a countable amenable group $G$, a F\o lner sequence $(F_N) \subseteq G$, and a set $E \subseteq G$ with $\bar{d}_{(F_N)}(E)=\limsup_{N \to \infty} \frac{|E \cap F_N|}{|F_N|}>0$, Furstenberg's correspondence principle associates with the pair $(E,(F_N))$ a measure preserving system $\mathds{X}=(X,\mathcal{B},\mu,(T_g)_{g \in G})$ and a set $A \in \mathcal{B}$ with $\mu(A)=\bar{d}_{(F_N)}(E)$, in such a way that for all $r \in \N$ and all $g_1,\dots,g_r \in G$ one has $\bar{d}_{(F_N)}(g_1^{-1}E \cap \dots \cap g_r^{-1}E)\geq\mu((T_{g_1})^{-1}A \cap \dots \cap (T_{g_r})^{-1}A)$. We show that under some natural assumptions, the system $\mathds{X}$ is unique up to a measurable isomorphism. We also establish variants of this uniqueness result for non-countable discrete amenable semigroups as well as for a generalized correspondence principle which deals with a finite family of bounded functions $f_1,\dots,f_{\ell}: G \rightarrow \C$. 
	\end{abstract}
	
	\section{Introduction}
	Szemer\'edi's celebrated theorem on arithmetic progressions states that any ``large'' set $E \subseteq \N$ in the sense that for some sequence of intervals $(I_N)$ with $|I_N| \to \infty$ as $N \to \infty$
	\begin{equation}\label{intro0} \bar{d}_{(I_N)}(E)=\limsup_{N \to \infty} \frac{|E \cap I_N|}{|I_N|}>0\end{equation}
	contains arbitrarily long arithmetic progressions, i.e. for all $k \in \N$ there is some $n \in \N$ such that $E \cap (E-n) \cap \dots (E-kn) \neq \emptyset $. In the seminal paper \cite{f1}, Furstenberg derived Szemer\'edi's theorem from his multiple recurrence result, which states that for any probability measure preserving system $(X,\mathcal{B},\mu,T)$, for any $k \in \N$ and for any $A \in \mathcal{B}$ with $\mu(A)>0$ there is some $n \in \N$ such that $\mu(A \cap T^{-n}A \cap \dots \cap T^{-kn}A)>0$. 
	\\ \\
	Furstenberg's derivation of Szemer\'edi's theorem from his multiple recurrence result hinges on a correspondence principle which allows one to associate with any ``large'' set $E$ a measure preserving system $(X,\mathcal{B},\mu,T)$ and a set $A \in \mathcal{B}$ with $\mu(A)>0$ such that $A \cap T^{-n}A \cap \dots T^{-kn}A \neq \emptyset$ implies $E \cap (E-n) \cap \dots \cap (E-kn) \neq \emptyset$.
	\\ \\
	We now describe Furstenberg's approach to creating such a measure preserving system. Given a ``large'' set $E$, he identifies $\mathbb{1}_E$ with a point $\omega=(\mathbb{1}_E(n))_{n \in \Z}$ in the symbolic space $\{0,1\}^{\Z}$. Now let $X=\overline{\{ T^n\omega : n \in \Z\}}$, where $T$ is the restriction of the shift map $\sigma((x_n)_{n \in \Z})=(x_{n+1})_{n \in \Z}$ to $X$, $\mathcal{B}_X$ be the completion of the Borel $\sigma$-algebra and $\mu$ be a weak* limit point of a sequence of measures $\frac{1}{|I_N|}\sum_{n \in I_N} \delta_{T^n\omega}$.
	\\ \\
	The system $(X,\mathcal{B}_X,\mu,T)$ satisfies the following natural conditions:
	\begin{itemize}
		\item[(1)] There is a set $A \in \mathcal{B}_X$ (namely, $A=\{x \in X: x(0)=1\}$) such that $\mu(A)=\bar{d}_{(I_N)}(E)>0$.
		\item[(2)] There is a subsequence $(I_{N_k})$ such that $\mu=\textrm{w*-}\lim_{k \to \infty} \frac{1}{|I_{N_k}|}\sum_{n \in I_{N_k}} \delta_{T^n\omega}$, which implies that for any $r \in \N$ and $h_1,\dots,h_r \in \Z$ we have
		\[ d_{(I_{N_k})}(E-h_1 \cap \dots \cap E-h_r)=\lim_{k \to \infty} \frac{|(E-h_1) \cap \dots \cap (E-h_r) \cap I_{N_k}|}{|I_{N_k}|}=\]
		\[\mu(T^{-h_1}A \cap \dots \cap T^{-h_r}A).\]
		\item[(3)] The $\sigma$-algebra $\mathcal{B}_X$ is generated by the family of sets $\{T^nA : n \in \Z\}$.
	\end{itemize}
	In principle, there are other ways to create a probability measure preserving system $(X,\mathcal{B},\mu,T)$ satisfying the conditions (1), (2) and (3) above (see Section 2). We will call any such system a \emph{Furstenberg system} associated with the pair $(E,(I_N))$.
	\\ \\
	The goal of this short paper is to show that any system satisfying (1), (2) and (3) is (metrically) isomorphic\footnote{In this paper ``isomorphism'' means \emph{metric} isomorphism. We will also be using in Section 4 a weaker, boolean, form of isomorphism which we will refer to as \emph{conjugacy}.} to the symbolic Furstenberg system described above. Actually, we will do this in the natural framework of amenable groups. A countable amenable group $G$ can be conveniently defined via the notion of a F\o lner sequence. A sequence $(F_N)$ of finite non-empty subsets of a countable group $G$ is a \emph{(left) F\o lner sequence} if for all $g \in G$ 
	\[ \lim_{N \to \infty} \frac{|F_N \Delta gF_N|}{|F_N|}=0.\] 
	A countable group $G$ is \emph{amenable} if it admits a (left) F\o lner sequence\footnote{Every countable amenable group $G$ admits also right- and indeed two-sided F\o lner sequences (see Corollary 5.3 in \cite{namioka}). Throughout this paper we deal only with left F\o lner sequences and routinely omit the adjective "left".}. Given a F\o lner sequence $(F_N) \subseteq G$ and a set $E \subseteq G$ we write
	\[ \bar{d}_{(F_N)}(E)=\limsup_{N \to \infty} \frac{|E \cap F_N|}{|F_N|}. \]
	Before continuing our discussion we want to formulate a general version of Furstenberg's correspondence principle (see for example \cite{b1}). 
	\begin{theorem}[Furstenberg's Correspondence Principle]\label{fcorrespondence}
		Let $G$ be a countable amenable group. Let $E \subseteq G$ with $\bar{d}_{(F_N)}(E)>0$ for some F\o lner sequence $(F_N) \subseteq G$. Then there exist a probability measure preserving system $(X,\mathcal{B},\mu,(T_g)_{g \in G})$, a set $A \in \mathcal{B}$ with $\mu(A)=\bar{d}_{(F_N)}(E)$ and a subsequence $(F_{N_k})$ such that for any $r \in \N$ and $g_1, \dots, g_r \in G$ one has:
		\begin{multline}\label{intro2} d_{(F_{N_k})}(g_1^{-1}E \cap \dots \cap g_r^{-1}E)=\lim_{k \to \infty} \frac{|g_1^{-1}E\cap \dots \cap g_r^{-1}E \cap F_{N_k}|}{|F_{N_k}|}= \\\mu((T_{g_1})^{-1}A \cap \dots \cap (T_{g_r})^{-1}A).
		\end{multline}
	\end{theorem}
	Theorem \ref{fcorrespondence} justifies the following definition:
	\begin{definition}\label{fsystem}
		Let $G$ be a countable amenable group, let $(F_N) \subseteq G$ be a F\o lner sequence, and let $E \subseteq G$ with $\bar{d}_{(F_N)}(E)>0$. We say that a standard\footnote{This means that the probability space $(X,\mathcal{B},\mu)$ is a Lebesgue space, so, in particular, it is a complete measure space (see Definition 2.3 in \cite{walters}).} probability measure preserving system $(X,\mathcal{B},\mu,(T_g)_{g \in G})$ is a \emph{Furstenberg system} associated with the pair $(E,(F_N))$ if
		\begin{itemize}
			\item[(1)] There is a set $A \in \mathcal{B}$ such that $\mu(A)=\bar{d}_{(F_N)}(E)>0$.
			\item[(2)] There is a subsequence $(F_{N_k})$ such that for all $r \in \N$ and all $g_1,\dots,g_r \in G$ we have
			\[ d_{(F_{N_k})}(g_1^{-1}E \cap \dots \cap g_r^{-1}E)=\mu((T_{g_1})^{-1}A \cap \dots \cap (T_{g_r})^{-1}A).\]
			\item[(3)] The $\sigma$-algebra $\mathcal{B}$ is (the completion of) the $\sigma$-algebra generated by the family of sets $\{T_gA : g \in G\}$.
		\end{itemize}
	\end{definition}
	
	\begin{theorem}\label{mainthm}
		Let $G$ be a countable amenable group and let $(F_N) \subseteq G$ be a F\o lner sequence. Let $E \subseteq G$ with $\bar{d}_{(F_N)}(E)>0$ and let $\mathds{X}=(X,\mathcal{B},\mu,(T_g)_{g \in G})$ be a Furstenberg system for the pair $(E,(F_N))$. Then, there exists a subsequence $(F_{N_k})$ such that the system $\mathds{X}$ is isomorphic to the symbolic system $\mathds{Y}=(Y,\mathcal{B}_Y,\nu,(S_g)_{g \in G})$, where $Y=\overline{\{S_g\omega : g \in G\}}$, $\omega=(\mathbb{1}_E(g))_{g \in G}$, the maps $S_g$ are given by $(S_gx)_{g_0}=x_{gg_0}$ for all $g, g_0 \in G$,  $\nu=\textrm{w*-}\lim_{k \to \infty} \frac{1}{|F_{N_k}|}\sum_{g \in F_{N_k}} \delta_{S_g\omega}$, and $\mathcal{B}_Y$ is the completion of the Borel $\sigma$-algebra of $Y$
	\end{theorem}
	While in Theorem \ref{mainthm} we choose to work, for the sake of simplicity, with countable amenable groups $G$, the proof extends without major modifications to countable cancellative amenable semigroups (they possess F\o lner sequences). In particular, Theorem \ref{mainthm} is valid for, say, $(\N,+)$ and $(\N, \cdot)$. One can actually extend the framework to an even more general setup. A discrete (not necessarily countable) group $G$ is called \emph{amenable} if there exists a left-invariant mean on $\ell^{\infty}(G)$. Call a set $E \subseteq G$ \emph{large} if, for some invariant mean $m$, we have $m(\mathbb{1}_E)>0$. One can modify Definition \ref{fsystem}, introduce a notion of Furstenberg system associated with the pair $(E,m)$ and then establish a general version of Theorem \ref{mainthm} (see Section 4). 
	
	The method of constructing the symbolic Furstenberg system described above for the case $G=\Z$ works equally well if one replaces the indicator function $\mathbb{1}_E$ with any bounded function $f: G \rightarrow \C$. Moreover, one can actually work with any finite family of bounded $\C$-valued functions (see for example \cite{fh1}). We will show in the Appendix that Theorem \ref{mainthm} can be naturally extended to this setup as well.
	\\ \\
	The structure of the paper is as follows. In Section 2 we describe four ways to construct a Furstenberg system. In Section 3 we prove Theorem \ref{mainthm}. In Section 4 we obtain a counterpart for Theorem \ref{mainthm} for not necessarily countable groups $G$. We also establish necessary and sufficient conditions for two pairs $(E_1,(F_N))$ and $(E_2,(G_N))$ to admit isomorphic Furstenberg systems. In Section 5 we prove that, given an ergodic measure preserving system $(X,\mathcal{B},\mu,(T_g)_{g \in G})$ and a set $A \in \mathcal{B}$ with $\mu(A)>0$, there is a set $E \subseteq G$ and a F\o lner sequence $(F_N)$ such that $\bar{d}_{(F_N)}(E)=\mu(A)$ and for which \eqref{intro2} holds. Finally, the Appendix deals with Furstenberg systems associated with a finite family of bounded functions.
	\begin{notation*} Throughout this paper, given a topological space $X$, we will routinely denote by $\mathcal{B}_X$ the completion of the Borel $\sigma$-algebra of $X$.
	\end{notation*}
	\section{Constructing a Furstenberg system}
	The purpose of this section is to describe four natural approaches to constructing a Furstenberg system $(X,\mathcal{B},\mu,(T_g)_{g \in G})$ associated with a ``large'' set $E \subseteq G$  (i.e. a set with $\bar{d}_{(F_N)}(E)>0$ for some F\o lner sequence $(F_N) \subseteq G$). The first of these approaches is a version for general amenable groups of Furstenberg's original construction (for $G=\Z)$, based on the symbolic space $Y=\{0,1\}^G$, which was reviewed in the introduction; we present it in this section for the sake of completeness. All four of the presented approaches have a unifying thread, which involves the usage (either implicitly or explicitly) of Riesz's representation theorem for positive linear functionals.
	\begin{itemize}
		\item[(i)] (cf. \cite{f1}). We consider the symbolic space $Y=\{0,1\}^G$ and proceed as follows. First, let $(F_{N_k})$ be a subsequence of $(F_N)$ such that $d_{(F_{N_k})}(E)=\bar{d}_{(F_N)}(E)>0$. Then, letting $\omega=(\mathbb{1}_E(g))_{g \in G}$, we consider a sequence of measures
		\begin{equation}
		\frac{1}{|F_{N_k}|}\sum_{g \in F_{N_k}} \delta_{S_g\omega}.
		\end{equation}
		Now, the space of probability measures on $\{0,1\}^G$ is a compact, metric space with respect to the weak* topology. Passing, if necessary, to a subsequence, let
		\[ \nu=\lim_{k \to \infty} \frac{1}{|F_{N_k}|}\sum_{g \in F_{N_k}} \delta_{S_g\omega}.\]
		Let $A=\{x : x(e)=1\}$. We have $\bar{d}_{(F_N)}(E)=d_{(F_{N_k})}(E)=\nu(A)$,  and
		\[  \nu((S_{g_1})^{-1}A \cap \dots \cap (S_{g_r})^{-1}A)=d_{(F_{N_k})}(g_1^{-1}E \cap \dots \cap g_r^{-1}E)\]
		for all $r \in \N$ and all $g_1,\dots,g_r \in G$. Observe now that the family $\{S_gA : g \in G\}$ generates the Borel $\sigma$-algebra of $Y$. Therefore, the completion of the $\sigma$-algebra generated by $A$ is equal to $\mathcal{B}_Y$, and thus the system $\mathds{Y}=(Y,\mathcal{B}_Y,\nu,(S_g)_{g \in G})$ satisfies (1), (2) and (3).
		\item[(ii)] In the presentation of the second approach (which actually was hinted at in \cite{f1}) we follow \cite{b1}. Since the family of sets of the form $\bigcap_{i=1}^r g_i^{-1}E,\  r \in \N, g_1,\dots,g_r \in G$, is countable, we can find a F\o lner subsequence $(F_{N_k})$ such that $d_{(F_{N_k})}(E)=\bar{d}_{(F_N)}(E)>0$ and
		\begin{equation}\label{limitfunctional}\lim_{k \to \infty}\frac{ \left|\bigcap_{i=1}^r g_i^{-1}E \cap F_{N_k}\right|}{|F_{N_k}|}=d_{(F_{N_k})}\left( \bigcap_{i=1}^r g_i^{-1}E  \right)
		\end{equation}
		exists for all $r \in \N$ and $g_1,\dots,g_r \in G$.
		\\ \\
		Let $\mathcal{A}=\textrm{Span}_{\C}\{\mathbb{1}, \prod_{i=1}^r \mathbb{1}_{g_i^{-1}E} : r \in \N, g_1,\dots,g_r \in G\}$, a vector subspace of $\ell^{\infty}(G)$. Equation \eqref{limitfunctional} allows us to define a functional $L: \mathcal{A} \rightarrow \C$ by setting $L(\mathbb{1})=1$ and $L\left( \prod_{i=1}^r \mathbb{1}_{g_i^{-1}E}\right)=d_{(F_{N_k})}\left( \bigcap_{i=1}^r g_i^{-1}E \right)$ and extending by linearity.
		\\ \\
		One can easily check that $|L(f)| \leq ||f||_{\ell^{\infty}(G)}$ for all $f \in \mathcal{A}$, so we can in fact extend $L$ by continuity to $\mathcal{A}_E$, the closure of $\mathcal{A}$ with respect to the $\ell^{\infty}(G)$-norm. Notice that $\mathcal{A}_E$ is a separable, unital $C^*$-algebra. By Gelfand's representation theorem, there is a compact metric space $X$ so that $\mathcal{A}_E$ is isometrically isomorphic to $C(X)$. Let $\Gamma: \mathcal{A}_E \longrightarrow C(X)$ be the corresponding isomorphism. 
		\\ \\
		The functional $L$ induces a positive linear functional $\tilde{L}: C(X) \rightarrow \C$. By Riesz's representation theorem, there exists a regular Borel probability measure $\mu$ on the Borel $\sigma$-algebra of $X$ such that for any $\varphi \in \mathcal{A}_E$
		\[L(\varphi)=\tilde{L}(\Gamma(\varphi))=\int_X \Gamma(\varphi) \ d\mu,\]
		Since $\mathbb{1}_E$ is an idempotent, and since $\Gamma$ is, in particular, an algebraic isomorphism, then $\Gamma(\mathbb{1}_E)=\mathbb{1}_A$, where $A \in \textrm{Borel}(X)$ and satisfies $\mu(A)=\tilde{L}(\mathbb{1}_A)=\tilde{L}(\Gamma(\mathbb{1}_E))=\bar{d}_{(F_N)}(E)$. 
		\\ \\
		The shift operators given by $\varphi(h) \mapsto \varphi(gh)$, $\varphi \in \mathcal{A}_E, g \in G$ induce a $G$-antiaction on $C(X)$, and are, by a theorem of Banach, induced by homeomorphisms $T_g: X \rightarrow X$. Using again the fact that $\Gamma$ is an algebraic isomorphism we have for all $r \in \N$, and $g_1,\dots,g_r \in G$
		\begin{equation}\label{measurepreservinghomeo} \mu((T_{g_1})^{-1}A \cap \dots \cap (T_{g_r})^{-1}A)=d_{(F_{N_k})}(g_1^{-1}E \cap \dots \cap g_r^{-1}E). \end{equation}
		(Note that by \eqref{measurepreservinghomeo} the homeomorphisms $T_g, g \in G$ preserve the measure $\mu$). Clearly $\mathds{X}=(X,\mathcal{B}_X,\mu,(T_g)_{g \in G})$ satisfies (1), (2) and (3).
	\end{itemize}
	In order to describe the third approach we need some facts on left invariant means. 
	Recall that a discrete semigroup $G$ is left amenable if there exists a left invariant mean $m: \ell^{\infty}(G) \rightarrow \C$\footnote{We say that $m \in \ell^{\infty}(G)^*$ is a left \emph{invariant mean} if it is a continuous linear functional from $\ell^{\infty}(G)$ to $\C$ such that (i) for every $f \in \ell^{\infty}(G)$ and for every $g \in G$ we have $m({}_gf)=m(f)$, where ${}_gf(x):=f(gx)$ for all $x \in G$, (ii) $m(f) \geq 0$ for any non-negative function $f: G \rightarrow \C$, and (iii) $m(1)=1$.}. Note that, for discrete countable groups, this is equivalent to the definition of left amenability given in the Introduction.
	\begin{itemize}
		\item[(iii)] We follow \cite{bl1} and work with $G$-invariant means directly. Namely, let $m: \ell^{\infty}(G) \rightarrow \C$ be a left-invariant mean on $G$. Take $X=\beta G$, the Stone-\v{C}ech compactification of $G$. Use Riesz's representation theorem to get the unique measure $\mu$ corresponding to $m$, so that $m(f)=\int_X \hat{f} \ d\mu$, where the function $\hat{f}$ is the extension by continuity to $\beta G$ of $f \in \ell^{\infty}(G)$. The maps $h \mapsto gh$, $h \in G$ have a unique continuous extension to $\beta G$ which we denote by $T_g$. Since $m$ is a $G$-invariant mean, we have that the maps $T_g$ are measure preserving homeomorphisms of $\beta G$.
		\\ \\
		Let $A=\overline{E}$, where $\bar{E}$ denotes the closure of $E$ in $\beta G$. Observe that
		\[ \mu(A \cap (T_{g_1})^{-1}A \cap \dots \cap (T_{g_r})^{-1}A)=m(\mathbb{1}_E \cdot \prod_{i=1}^r \mathbb{1}_{g_i^{-1}E}),\]
		Let $\mathcal{C}_A$ be the restriction of the $\sigma$-algebra $\mathcal{B}_X$ to the completion of the $\sigma$-algebra generated by $A$. Then, the system $\mathds{X}=(X,\mathcal{C}_A,\mu,(T_g)_{g \in G})$ satisfies (1), (2) and (3).
		\item[(iv)] The last approach we present follows \cite{bmc} and has elements in common with (i) and (iii). Let $Y=\{0,1\}^G$ and consider the cylinders of the form
		\begin{equation}\label{cylinder} \{ x \in Y: x(h_1)=\varepsilon_1,\dots, x(h_r)=\varepsilon_r \},\end{equation}
		where $r \in \N$, and $h_1,\dots,h_r \in G$ are distinct and $\varepsilon_i \in \{0,1\}$. For any cylinder $C$ of the form \eqref{cylinder}, put
		\[ \lambda(C)=m\left(\prod_{i=1}^r \mathbb{1}_{h_i^{-1}E_i}\right),\]
		where $E_i$ is equal to $E$ or $E^c$ according to whether $\varepsilon_i=1$ or $\varepsilon_i=0$, respectively. Since $Y$ is a compact metric space, the premeasure $\lambda$ extends to a measure $\mu$ on the Borel $\sigma$-algebra of $Y$. One can easily check that $\mu$ is invariant under the shifts $S_g$ as in (i). Letting $A=\{x : x(e)=1\}$ we get
		\begin{equation}\label{determinemu}
		\mu((T_{g_1})^{-1}A \cap \dots \cap (T_{g_r})^{-1}A)=m\left(\prod_{i=1}^r \mathbb{1}_{g_i^{-1}E}\right)
		\end{equation}
		for all $r \in \N$ and $g_1,\dots,g_r \in G$. Moreover, the measure $\mu$ is determined by the intersections appearing in \eqref{determinemu} above, so the completion of the $\sigma$-algebra generated by $A$ is equal to $\mathcal{B}_Y$, and the system $\mathds{Y}=(Y,\mathcal{B}_Y,\mu,(S_g)_{g \in G})$ satisfies (1), (2) and (3).
	\end{itemize}
	\section{Theorem \ref{mainthm} and a combinatorial corollary.}
	The main purpose of this section is to prove Theorem \ref{mainthm} that characterizes Furstenberg systems for countable groups up to a measurable isomorphism. In particular, we show that any Furstenberg system is isomorphic to the symbolic measure preserving system $\mathds{Y}$ described in item (i) of Section 2. As a corollary we obtain a criterion which determines when the Furstenberg systems of two pairs $(E,(F_N)), (D,(G_N))$ are isomorphic (here $E, D$ are subsets of $G$ and $(F_N), (G_N)$ are F\o lner sequences).
	\\ \\
	In the following two sections and in the Appendix, we will make repeated use of the following notation: if $\mathcal{A}$ is a commutative, unital $C^*$-algebra, then $\hat{\mathcal{A}}$ denotes the space of characters of $\mathcal{A}$, i.e. the space of algebra homomorphisms $\chi: \mathcal{A} \rightarrow \C$. Recall that $\hat{\mathcal{A}}$ is a compact Hausdorff space with respect to the weak* topology (in fact, a metric space if the algebra $\mathcal{A}$ is separable).
	\\ \\
	We start with the following technical result which will be needed in the sequel:
	\begin{theorem}[cf. Satz 1 in \cite{vonneumann} and Remark 4.1 in \cite{petersen}]\label{peterseniso}
		Consider two standard measure preserving systems $\mathds{X}=(X,\mathcal{B}_X,\mu,(T_g)_{g \in G})$ and $\mathds{Y}=(Y,\mathcal{B}_Y,\nu,(S_g)_{g \in G})$. Let $\Phi: L^{\infty}(X,\mathcal{B}_X,\mu) \rightarrow L^{\infty}(Y,\mathcal{B}_Y,\nu)$ be a Banach algebra isomorphism satisfying 
		\[ \Phi(f \circ T_g)=\Phi(f) \circ S_g \quad \textrm{ and } \int_Y \Phi(f) \ d\nu=\int_X f \ d\mu \textrm{ for all } g \in G, f \in L^{\infty}(X,\mathcal{B}_X,\mu).\]
		Then, there exists an isomorphism between $\mathds{X}$ and $\mathds{Y}$.
	\end{theorem}
	\begin{proof}
		Arguing as in the proof of Remark 4.1 in \cite{petersen} one deduces that the measure algebras $(\mathcal{B}_X,\mu)$ and $(\mathcal{B}_Y,\nu)$ are isomorphic (as Boolean algebras). Now Satz 1 in \cite{vonneumann} (see Theorem 1.4.6 of \cite{petersen}) implies that $\Phi$ arises from a point isomorphism, say $\tilde{\Phi}: (X,\mathcal{B}_X,\mu) \longrightarrow (Y,\mathcal{B}_Y,\nu)$. Finally, the condition $\Phi(f \circ T_g)=\Phi(f) \circ S_g$, $g \in G, f \in L^{\infty}(X,\mathcal{B}_X,\mu)$ implies that for every $g \in G$, there is a set $X_g$ with $\mu(X_g)=1$ such that $S_g\tilde{\Phi}(x)=\tilde{\Phi}T_g(x)$, for $x \in X_g$. It follows that the restriction of $\tilde{\Phi}$ to the set $\bigcap_{g \in G} X_g$ is the desired point isomorphism (observe that since $G$ is countable, this intersection is a measurable set with full measure).
	\end{proof}
	We now give a proof of Theorem \ref{mainthm} which we state here again for the convenience of the reader:
	\begin{theorem}\label{mainthmproof} Let $G$ be a countable amenable group and let $(F_N) \subseteq G$ be a F\o lner sequence. Let $E \subseteq G$ with $\bar{d}_{(F_N)}(E)>0$ and let $\mathds{X}=(X,\mathcal{B},\mu,(T_g)_{g \in G})$ be a Furstenberg system for the pair $(E,(F_N))$. Then, there exists a subsequence $(F_{N_k})$ such that the system $\mathds{X}$ is isomorphic to the symbolic system $\mathds{Y}=(Y,\mathcal{B}_Y,\nu,(S_g)_{g \in G})$, where $Y=\overline{\{S_g\omega : g \in G\}}$, $\omega=(\mathbb{1}_E(g))_{g \in G}$, the maps $S_g$ are given by $(S_gx)_{g_0}=x_{gg_0}$ for all $g, g_0 \in G$,  $\nu=\textrm{w*-}\lim_{k \to \infty} \frac{1}{|F_{N_k}|}\sum_{g \in F_{N_k}} \delta_{S_g\omega}$, and $\mathcal{B}_Y$ is the completion of the Borel $\sigma$-algebra of $Y$\footnote{Note that since $\textrm{supp}(\nu) \subseteq \overline{\{S_g\omega : g \in G\}}$, we could take $Y=\{0,1\}^G$.}.
	\end{theorem}
	\begin{proof}
		Let $f=\mathbb{1}_A$, and put $\mathcal{A}_f=\overline{\textrm{Span}_{\C}\{ \prod_{i=1}^r T_{g_i}f, \mathbb{1} : r \in \N, g_1,\dots,g_r \in G\}}^{L^{\infty}(\mu)}$, the unital $C^*$-algebra generated by products of shifts of $f$. Since $\mathds{X}$ is a Furstenberg system, it satisfies condition (3) in Definition \ref{fsystem}, which implies that $\mathcal{A}_f$ is dense in $L^{\infty}(\mathcal{B},\mu)$ with respect to the $L^1(\mu)$-norm.
		\\ \\
		Let $Z=\hat{\mathcal{A}_f}$. By Gelfand's representation theorem, there is an isometric isomorphism $\Gamma: \mathcal{A}_f \rightarrow C(Z)$. Notice $Z$ is compact and metric given that $\mathcal{A}_f$ is separable. The measure $\mu$ induces a positive linear functional $F: C(Z) \rightarrow \C$ via $F(\varphi)=\mu(\Gamma^{-1}(\varphi))$, where $\varphi \in C(Z)$. Let $\tilde{\mu}$ be the measure obtained from $F$ via Riesz's representation theorem.  Each measure preserving transformation $T_g, g \in G$, induces an algebra isomorphism of $\mathcal{A}_f$ (via precomposition), which in turn, induces an isomorphism of $C(Z)$ via $\Gamma$. A theorem of Banach ensures that this isomorphism of $C(Z)$ is given by a homeomorphism of $Z$, which we will denote by $\tilde{T_g},g \in G$. Observe that the homeomorphism $\tilde{T_g}, g \in G$ is $\tilde{\mu}$-preserving. We claim that the system $\mathds{X}=(X,\mathcal{B},\mu,(T_g)_{g \in G})$ is isomorphic to the system $\mathds{Z}=(Z,\mathcal{B}_Z,\tilde{\mu},(\tilde{T_g})_{g \in G})$.
		\\ \\
		Note that since $\mathcal{A}_f$ is dense, with respect to the $L^1(\mu)$-norm, in $L^{\infty}(\mathcal{B},\mu)$, and since $C(Z)$ is dense, with respect to the $L^1(\tilde{\mu})$-norm, in $L^{\infty}(Z,\mathcal{B}_Z,\tilde{\mu})$, we can extend $\Gamma$ by continuity to a Banach algebra isomorphism $\tilde{\Gamma}$ between $L^{\infty}(X,\mathcal{B},\mu)$ and $L^{\infty}(Z,\mathcal{B}_Z,\tilde{\mu})$. Moreover, it is not hard to see that the map $\tilde{\Gamma}$ satisfies
		\[ \tilde{\Gamma}(f \circ T_g)=\tilde{\Gamma}(f) \circ \tilde{T}_g \quad \textrm{ and } \int_Z \tilde{\Gamma}(f) \ d\tilde{\mu}=\int_X f \ d\mu \textrm{ for all } g \in G, f \in L^{\infty}(X,\mathcal{B},\mu).\]
		(Indeed, it is enough to check that the above equalities  hold for functions of the form $\prod_{i=1}^r T_{g_i}f$, with $r \in \N$ and $g_1,\dots,g_r \in G$). 
		\\ \\
		Thus, by Theorem \ref{peterseniso}, there exists an isomorphism between the measure preserving systems $\mathds{X}$ and $\mathds{Z}$. From this point on we are going to keep working with the system $\mathds{Z}$.
		\\ \\
		Define a map $\Phi$ from $Z$ to the compact space $\{0,1\}^G$ by $\Phi(\chi)=(\chi(T_gf))_{g \in G}$. Note that $\Phi$ is well defined. Indeed, since $\chi \in Z$, and $T_gf$ is idempotent for all $g \in G$, it follows that $\chi(T_gf) \in \{0,1\}$. 
		\\ \\
		Next, we show that $\Phi$ is injective. Suppose that $\Phi(\chi_1)=\Phi(\chi_2)$. Then, $\chi_1(T_gf)=\chi_2(T_gf)$ for all $g \in G$. Since $\chi_1,\chi_2$ are multiplicative, we have $\chi_1\left(\prod_{i=1}^r T_{g_i}f\right)=\chi_2\left(\prod_{i=1}^r T_{g_i}f\right)$ for all $r \in \N$, and all $g_1,\dots,g_r \in G$. Since $\chi_1,\chi_2$ are linear, we see that $\chi_1(g)=\chi_2(g)$ for any $g \in \textrm{Span}_{\C}\left\{ \prod_{i=1}^r T_{g_i}f : r \in \N, g_i \in G\right\}$. Finally, since $\chi_1,\chi_2$ are continuous, it follows that $\chi_1(g)=\chi_2(g)$ for any $g \in \mathcal{A}_f$, whence $\chi_1=\chi_2$, which implies that $\Phi$ is injective.
		\\ \\
		Finally, the map $\Phi$ is continuous. Indeed, for each $g \in G$, the map $\chi \mapsto \chi(T_gf)$ is an evaluation map, and as such, is continuous with respect to the weak* topology for all $g \in G$. This implies that $\Phi$ is continuous (we consider $\{0,1\}^G$ with respect to its natural product topology). Notice that, in particular, this means that $\Phi$ is measurable.
		\\ \\
		Now, let $\nu=\Phi_{*}\mu$. The map $\Phi$ provides a measurable isomorphism between the measure preserving systems $\mathds{Z}$ and $(\{0,1\}^G,\mathcal{B}_{\{0,1\}^G},\nu,(S_g)_{g \in G})$, where $S_g, g \in G$ denote the shift maps. Letting $B=\{x \in \{0,1\}^G: x(e)=1\}$, we see that the probability measure $\nu$ is determined by the values
		\begin{equation}
		\nu((T_{g_1})^{-1}B \cap \dotso \cap (T_{g_r})^{-1}B), \textrm{ for all } r \in \N, \ g_1,\dots,g_r \in G.
		\end{equation}
		We have
		\begin{multline}\label{densityormean}
		\nu((T_{g_1})^{-1}B \cap \dotso \cap (T_{g_r})^{-1}B)=\int_Y \mathbb{1}_{(T_{g_1})^{-1}B \cap \dotso \cap (T_{g_r})^{-1}B} \ d\Phi_{*}\tilde{\mu}=\\ \int_Z \mathbb{1}_{(T_{g_1})^{-1}B \cap \dotso \cap (T_{g_r})^{-1}B}(\Phi(\chi)) \ d\tilde{\mu}(\chi)=\\ \int_Z \mathbb{1}_{(T_{g_1})^{-1}B \cap \dotso \cap (T_{g_r})^{-1}B}((\chi(T_gf))_{g \in G}) \ d\tilde{\mu}(\chi)= \int_Z \chi(T_{g_1}f \cdot \dotso \cdot T_{g_r}f) \ d\tilde{\mu}(\chi)=\\ \int_Z \Gamma(T_{g_1}f \cdot \dotso \cdot T_{g_r}f) \ d\tilde{\mu}(\chi)=\int_X T_{g_1}f \cdot \dotso \cdot T_{g_r}f \ d\mu= \\ \mu((T_{g_1})^{-1}A \cap \dots \cap (T_{g_r})^{-1}A)=d_{(F_{N_k})}(g_1^{-1}E \cap \dotso \cap g_r^{-1}E),
		\end{multline}
		where we used the definition of the Gelfand transform in the second to last step, and the fact that $\mathds{X}$ is a Furstenberg system in the last one. It follows that, for $\omega=(\mathbb{1}_E(g))_{g \in G}$, one has
		\begin{equation}\label{weak1}
		\nu=\textrm{w*}\lim_{k \to \infty} \frac{1}{|F_{N_k}|}\sum_{g \in F_{N_k}} \delta_{S_g\omega},
		\end{equation}
		whence $\textrm{supp}(\nu) \subseteq \overline{\{ S_g\omega : g \in G\}}$. Notice that a reasoning as in (i) of Section 2 shows that the right hand side in \eqref{weak1} is well defined as a weak* limit. We are done.
	\end{proof}
	Next, we make some comments on ergodicity of Furstenberg systems. In general, a Furstenberg system associated with a pair $(E,(F_N))$ will not be ergodic (this is for example the case for $(E=\bigcup_{n \in \N}[2^{2n},2^{2n+1}),F_N=[1,N])$.) Nonetheless, one can show (see \cite{bf}) that if for some F\o lner sequence $(F_N)$ the set $E$ satisfies $d_{(F_N)}(E)>0$, then there is a F\o lner sequence $(G_N)$ such that the Furstenberg system associated to $(E,(G_N))$ is ergodic.
	\\ \\
	The following corollary of Theorem \ref{mainthmproof} provides a characterization of pairs $(E,(F_N))$, $(D,(G_N))$ that have isomorphic Furstenberg systems:
	\begin{theorem}\label{setchar}
		Let $E, D$ be two subsets of $G$ and let $(F_N), (G_N)$ be F\o lner sequences such that $d_{(F_N)}(E)>0$ and $d_{(G_N)}(D)>0$.
		\begin{itemize}
			\item[(1)] If the Furstenberg systems associated with $(E,(F_N))$ and $(D,(G_N))$ are isomorphic, then, for all $r \in \N$ and all $g_1,\dots,g_r \in G$,
			\begin{equation}\label{setequivalence}
			d_{(F_N)}(g_1^{-1}E\cap \dots \cap g_r^{-1}E)=d_{(G_N)}(g_1^{-1}D \cap \dots \cap g_r^{-1}D).
			\end{equation}
			\item[(2)] Suppose that the pairs $(E,(F_N))$ and $(D,(G_N))$ satisfy \eqref{setequivalence}. Then, $(E,(F_N))$ and $(D,(G_N))$ admit isomorphic Furstenberg systems.
		\end{itemize}
	\end{theorem}
	\begin{proof}
		(1) Let us assume that the pairs $(E,(F_N))$ and $(D,(G_N))$ satisfy \eqref{setequivalence}. By Theorem \ref{mainthmproof} the sets $E$ and $D$ have each a Furstenberg system of the form $(Y,\mathcal{B}_Y,\nu_1,(S_g)_{g \in G})$ and $(Y,\mathcal{B}_Y,\nu_2,(S_g)_{g \in G})$, where $Y=\{0,1\}^G$ and where, for $\omega_1=(\mathbb{1}_{E}(g))_{g \in G}$ and $\omega_2=(\mathbb{1}_{D}(g))_{g \in G}$, we put $\nu_1=\textrm{w*-}\lim_{N \to \infty} \frac{1}{|F_N|}\sum_{g \in F_N}\delta_{S_g\omega_1}$ and $\nu_2=\textrm{w*-}\lim_{N \to \infty} \frac{1}{|G_N|}\sum_{g \in G_N}\delta_{S_g\omega_2}$ (we use $Y=\{0,1\}^G$ instead of the corresponding orbital closures for convenience -see footnote 5).
		
		As we know from (i) in Section 2, the measures $\nu_1$ and $\nu_2$ are determined by their values at sets of the form $S_{g_1}^{-1}A \cap \dots \cap S_{g_r}^{-1}A$ for $r \in \N$, $g_1,\dots,g_r \in G$, where $A=\{x : x(e)=1\}$. The result in question follows now from \eqref{setequivalence}.
		\\ \\
		(2) If the pairs $(E,(F_N)$ and $(D,(G_N))$ admit isomorphic Furstenberg systems, say $\mathds{Y}_i=(Y,\mathcal{B}_Y,\nu_i,(S_g)_{g \in G})$, for $i=1,2$, it follows that $\nu_1((S_{g_1})^{-1}A \cap \dots \cap (S_{g_r})^{-1}A)=\nu_2( (S_{g_1})^{-1}A \cap \dots \cap (S_{g_r})^{-1}A)$, for all $r \in \N$ and $g_1,\dots,g_r \in G$, which implies, given the definition of $\nu_1, \nu_2$, that we can find F\o lner sequences $(F_N)$ and $(G_N)$ such that \eqref{setequivalence} holds, as desired.
	\end{proof}
	\section{A version of Theorem \ref{mainthm} when $G$ is uncountable}
	Since there are quite a few results of Ramsey-theoretical nature which are valid for uncountable amenable semigroups (see for example \cite{bl1}, \cite{hs} and \cite{dl}), it makes sense to consider a variant of Theorem \ref{mainthm} for general (discrete) amenable semigroups. In what follows, we will be assuming (without loss of generality) that $G$ has a neutral element. 
	\begin{definition}\label{fsystem2}
		Let $G$ be a left amenable semigroup. Let $m: \ell^{\infty}(G) \rightarrow \C$ be a left-invariant mean and let $E \subseteq G$ be such that $m(\mathbb{1}_E)>0$. We say that a measure preserving system $(X,\mathcal{B},\mu,(T_g))$ is a \emph{Furstenberg system} associated with the pair $(E,m)$ if
		\begin{itemize}
			\item[(1)] There is a set $A \in \mathcal{B}$ such that $\mu(A)=m(\mathbb{1}_E)>0$.
			\item[(2)] For all $r \in \N$ and all $g_1,\dots,g_r \in G$ we have $m(\mathbb{1}_{g_1^{-1}E} \cdot \dotso \cdot \mathbb{1}_{g_r^{-1}E})=\mu((T_{g_1})^{-1}A \cap \dots \cap (T_{g_r})^{-1}A)$.
			\item[(3)] The $\sigma$-algebra $\mathcal{B}$ is (the completion of) the $\sigma$-algebra generated by the family of sets $\{T_gA : g \in G\}$.
		\end{itemize}
	\end{definition}
	Next we describe a ``symbolic'' Furstenberg system associated to a pair $(E,m)$.
	\begin{lemma}\label{seqspace2}
		Let $G$ be a left amenable semigroup and let $E \subseteq G$ be such that for some left-invariant mean $m$ on $G$ we have $m(\mathbb{1}_E)>0$. Put $X=\{0,1\}^G$ and let $A=\{x : x(e)=1\}$. Then, the system $\mathds{X}=(X,\mathcal{B}_X,\mu,(S_g)_{g \in G})$, where $S_g$ is given by $(S_gx)_{g_0}=x_{gg_0}$ for all $g, g_0 \in G$, and $\mu$ is a (unique) measure satisfying
		\begin{equation}\label{sw1}
		\mu((S_{g_1})^{-1} A \cap \dots \cap (S_{g_r})^{-1}A)=m(\mathbb{1}_{g_1^{-1}E} \cdot \dotso \cdot \mathbb{1}_{g_r^{-1}E})
		\end{equation}
		for all $r \in \N$ and $g_1,\dots,g_r \in G$, satisfies (1), (2) and (3) of Definition \ref{fsystem2}.  
	\end{lemma}
	\begin{proof}
		We start by noting that $\mu$ is well defined (this follows, for example, from the proof of Theorem 2.1 in \cite{bmc}, given that countability of $G$ is not used there).
		Now we observe that uniqueness of $\mu$ follows from the Stone-Weierstrass theorem. Indeed, the unital $\C$-algebra generated by the functions $\{ \mathbb{1}_{(S_g)^{-1}A} : g \in G\}$ clearly separates points of $X$ and is closed under conjugation. Thus, by Stone-Weierstrass, it is dense in $C(X)$, and it follows that the values in \eqref{sw1} determine $\mu$.
		The system $\mathds{X}$ clearly satisfies (1) and (2) of Definition \ref{fsystem2} by \eqref{sw1}. Moreover, the $G$-invariant $\sigma$-algebra generated by $A$ contains $\textrm{Borel}(X)$, whence $\mathds{X}$ satisfies (3), completing the proof.
	\end{proof}
	The following variant of Theorem \ref{peterseniso} will be needed for the proof of Theorem \ref{mainthm2} below.
	\begin{theorem}[cf. Remark 4.1 \cite{petersen}]\label{peterseniso2}
		Let $(X,\mathcal{B},\mu,(T_g)_{g \in G})$ and $(Y,\mathcal{C},\nu,(S_g)_{g \in G})$ be two measure preserving systems. Suppose that we can find a Banach algebra isomorphism $\Phi: L^{\infty}(X,\mathcal{B},\mu) \rightarrow L^{\infty}(Y,\mathcal{C},\nu)$ satisfying 
		\[ \Phi(f \circ T_g)=\Phi(f) \circ S_g \quad \textrm{ and } \int_Y \Phi(f) \ d\nu=\int_X f \ d\mu \textrm{ for all } g \in G, f \in L^{\infty}(X,\mathcal{B},\mu).\]
		Then, $(T_g)_{g \in G}$ and $(S_g)_{g \in G}$ are conjugate (see Definition 2.5 in \cite{walters}).
	\end{theorem}
	\begin{proof}
		The argument goes along the lines of the first part of the proof of Theorem \ref{peterseniso}. Notice that we do not assume that $(X,\mathcal{B},\mu)$ or $(Y,\mathcal{C},\nu)$ are standard. This generality is compensated by the fact that the map $\Phi$ is just an algebra isomorphism, rather than a pointwise map.
	\end{proof}
	Here is finally a version of Theorem \ref{mainthm} for general amenable semigroups.
	\begin{theorem}\label{mainthm2}
		Let $G$ be a left amenable semigroup and let $m$ be a left-invariant mean. Let $E \subseteq G$ with $m(\mathbb{1}_E)>0$ and let $\mathds{X}=(X,\mathcal{B}_X,\mu,(T_g)_{g \in G})$ be a Furstenberg system for the pair $(E,m)$. Let $\mathds{Y}$ be the symbolic Furstenberg system for $(E,m)$ constructed in Lemma \ref{seqspace2}. Then, $(T_g)_{g \in G}$ and $(S_g)_{g \in G}$ are conjugate.
	\end{theorem}
	\begin{proof}
		The proof is essentially the same as that of Theorem \ref{mainthmproof}. Only two changes have to be made. First, instead of applying Theorem \ref{peterseniso}, we need to make use of Theorem \ref{peterseniso2}, since $G$ is not assumed to be countable. Second, we have to replace $d_{(F_{N_k})}$ in the formula
		$\mu((T_{g_1})^{-1}A \cap \dots \cap (T_{g_r})^{-1}A)=d_{(F_{N_k})}(g_1^{-1}E \cap \dotso \cap g_r^{-1}E)
		\textrm{ (see equation \eqref{densityormean})}$
		with the mean $m$ and the measure $\nu$ in \eqref{weak1} by the measure $\nu$ obtained in Lemma \ref{seqspace2}.
	\end{proof}
	We conclude this section with the observation that an analog of Theorem \ref{setchar} for pairs $(E,m_1)$ and $(D,m_2)$ holds for a general amenable semigroup $G$. We omit the details.
	\section{From a dynamical system back to the group}
	The purpose of this short section is to prove the following partial converse to Theorem \ref{mainthm}.
	\begin{theorem}\label{ergodicback}
		Let $G$ be a countable amenable group.  Let $\mathds{X}=(X,\mathcal{B},\mu,(T_g)_{g \in G})$ be an ergodic measure preserving system. Let $A \in \mathcal{B}$ with $\mu(A)>0$ and let $(F_N)$ be a F\o lner sequence in $G$. Then there exists a subsequence $(F_{N_k})$ and a set $E \subseteq G$ such that 
		\begin{equation}\label{ergback}
		\mu((T_{g_1})^{-1}A \cap \dots \cap (T_{g_k})^{-1}A)=d_{(F_{N_k})}(g_1^{-1}E \cap \dots \cap g_k^{-1}E)
		\end{equation}
		for all $k \in \N$ and $g_1,\dots,g_k \in G$ (in particular, $\mu(A)=d_{(F_{N_k})}(E)$).
	\end{theorem}
	\begin{proof} 
		Let $A \in \mathcal{B}$ with $\mu(A)>0$. Let $\mathcal{P}_f(G)$ denote the set of finite subsets of $G$. Put $F=\mathbb{1}_A$, and for every $\mathcal{G} \in \mathcal{P}_f(G)$, let $F_{\mathcal{G}}=\prod_{g \in \mathcal{G}} T_gF$. Note that the set $\{F_{\mathcal{G}} : \mathcal{G} \in \mathcal{P}_f(G) \}$ is countable. By von Neumann's mean ergodic theorem we have
		\begin{equation}
		\lim_{N \to \infty} \frac{1}{|F_N|}\sum_{g \in F_N} T_gF_{\mathcal{G}}=\int_X F_{\mathcal{G}} \ d\mu,
		\end{equation}
		for all $\mathcal{G} \in \mathcal{P}_f(G)$, where convergence takes place in the $L^2(\mu)$-norm. Since $\mathcal{P}_f(G)$ is countable, we can extract a subsequence $(F_{N_k})$ via a diagonal process and a subset $X_0 \subseteq X$ with $\mu(X_0)=1$ such that for all $x \in X_0$ and all $\mathcal{G} \in \mathcal{P}_f(G)$ we have
		\begin{equation}
		\lim_{k \to \infty} \frac{1}{|F_{N_k}|}\sum_{g \in F_{N_k}} F_{\mathcal{G}}(T_gx)=\int_X F_{\mathcal{G}} \ d\mu.
		\end{equation}
		Let $x \in X_0$, and let $E(x)=\{ g \in G : T_gx \in A\}$. Clearly, the set $E(x)$ satisfies \eqref{ergback}.
	\end{proof}
	\begin{corollary}
		Let $X_0$ be the subset of $X$ obtained in the proof of Theorem \ref{ergodicback}. For each $x \in X_0$, the Furstenberg system associated with $(E(x),(F_{N_k}))$ is isomorphic to a factor $\mathds{Z}_x$ of $\mathds{X}$ which can be identified with the completion of the $G$-invariant $\sigma$-algebra generated by the set $A$. Moreover, for any $x, y \in X_0$, the factors $\mathds{Z}_x$ and $\mathds{Z}_y$ are isomorphic.
	\end{corollary}
	\appendix
	\section{Furstenberg systems that arise from finitely many functions}
	Let $G$ be a countable amenable group. The purpose of this section is to introduce Furstenberg systems associated with finite families of functions $f_1,\dots,f_{\ell} : G \rightarrow \D$, where $\D$ is the closed unit disk in $\C$, and to establish results analogous to those in Section 3.
	\begin{definition}
		Let $G$ be a countable amenable group. Let $\ell \in \N$ and $f_1,\dots,f_{\ell}: G \rightarrow \D$. Assume that there exists a F\o lner sequence $(F_N) \subseteq G$ and functions $a: G \rightarrow \R_{\geq 0}$ and $b : \N \rightarrow \R_{\geq 0}$ satisfying $\frac{1}{b(N)|F_N|}\sum_{g \in F_N}a(g) \to 1$ as $N \to \infty$. We will call such a triple $((F_N),a,b)$ an averaging scheme. We say that the family $\{f_1,\dots,f_{\ell}\}$ is \emph{accordant with the averaging scheme}  $((F_N),a,b)$ if for any choice of $\tilde{f_i} \in \{f_1,\dots,f_{\ell},\bar{f_1},\dots,\bar{f_{\ell}}\}$ the limit
		\begin{equation}
		\lim_{N \to \infty} \frac{1}{b(N)|F_N|}\sum_{g \in F_N} a(g)\prod_{i=1}^r \tilde{f_i}(g_ig)
		\end{equation}
		exists for all $r \in \N$, $g_1,\dots,g_r \in G$. 
	\end{definition}
	We are in a position to define Furstenberg systems associated to a family of functions $\{f_1,\dots,f_{\ell}\}$ accordant with the averaging scheme $((F_N),a,b)$:
	\begin{definition}\label{fsystem3}
		Let $G$ be a countable amenable group. Let $((F_N),a,b)$ be an averaging scheme. Let $\ell \in \N$ and let $f_1,\dots,f_{\ell} :G \rightarrow \D$ be a family of functions accordant with $((F_N),a,b)$. We say that a standard measure preserving system $(X,\mathcal{B},\mu,(T_g))$ is a \emph{Furstenberg system} for the tuple $(\{f_1,\dots,f_{\ell}\},(F_N),a,b)$ if there are functions $F_1,\dots,F_{\ell} \in L^{\infty}(\mu)$ such that for any $r \in \N$ and any $g_1,\dots,g_r \in G$ and for any choice of $\tilde{f_i} \in \{ f_1,\dots,f_{\ell},\bar{f_1},\dots,\bar{f_{\ell}}\}$ and $\tilde{F_i} \in \{F_1,\dots,F_{\ell},\bar{F_1},\dots,\bar{F_{\ell}}\}$ (where the choice of $\tilde{f_i}$ and $\tilde{F_i}$ is made in a ``simultaneous'' manner) the limit 
		\begin{equation}
		\lim_{N \to \infty} \frac{1}{b(N)|F_N|}\sum_{g \in F_N}a(g) \prod_{i=1}^r \tilde{f_i}(g_ig)=\int_X T_{g_1}\tilde{F_1} \cdot \dotso \cdot T_{g_r}\tilde{F_r} \ d\mu
		\end{equation}
		exists, and moreover, the $\sigma$-algebra $\mathcal{B}$ is equal to the completion of the $G$-invariant $\sigma$-algebra generated by the measurable functions $F_1,\dots,F_{\ell}$.
	\end{definition}
	As in Section 2, one can construct a symbolic Furstenberg system for the tuple $(\{f_1,\dots,f_{\ell}\},(F_N),a,b)$:
	\begin{lemma}\label{sequencespace2}
		Let $G$ be a countable amenable group. Let $((F_N),a,b)$ be an averaging scheme. Let $\{f_1,\dots,f_{\ell}\}$ be a family of functions accordant with $((F_N),a,b)$. Let $Y=(\D^{\ell})^G$. Put $\mathds{Y}=(Y,\mathcal{B}_Y,\nu,(S_g)_{g \in G})$, where for $\omega=(f_1(g),\dots,f_{\ell}(g))_{g \in G}$, $\nu=\textrm{w*-}\lim_{N \to \infty} \frac{1}{b(N)|F_N|} \sum_{g \in F_N} a(g)\delta_{S_g\omega}$, and $(S_gx)_{g_0}=(x_1(gg_0),\dots,x_{\ell}(gg_0))$ for all $g, g_0 \in G$ (we will show that this weak* limit is well defined). Then, $\mathds{Y}$ is a Furstenberg system for the tuple $(\{f_1,\dots,f_{\ell}\},(F_N),a,b)$.
	\end{lemma}
	\begin{proof} Let $Y=(\D^{\ell})^G$ and put $F_i=x_i(0) \in C(Y)$, so $F_i$ is measurable with respect to $\mathcal{B}_Y$. Then one easily checks that $\mathcal{C}(F_1,\dots,F_{\ell})$, the $\sigma$-algebra generated by the functions $F_1,\dots,F_{\ell},\bar{F_1},\dots,\bar{F_{\ell}}$, is equal to $\mathcal{B}_Y$ (since we tacitly consider $Y$ with the product topology and because, by the Stone-Weierstrass theorem, the subalgebra generated by these functions is dense in $C(Y)$). Invoking the Stone-Weierstrass theorem again, we see that the measure $\nu$ is determined by the values
		\[ \int_Y \prod_{i=1}^r T_{g_i}G_i \ d \rho \textrm{ for all } r \in \N, g_1,\dots,g_r \in G \textrm{ and } G_i \in \{F_1,\dots,F_{\ell},\bar{F_1},\dots,\bar{F_{\ell}}\}.  \]
		Let $\nu=\textrm{w*-}\lim_{N \to \infty} \frac{1}{b(N)|F_N|}\sum_{g \in F_N} a(g)\delta_{S_g\omega}$, where $\omega=(f_1(g),\dots,f_{\ell}(g))_{g \in G}$. Observe that $\nu$ is well defined because 
		\[ \int_Y \prod_{i=1}^r T_{g_i}\tilde{F_i} \ d\nu=\lim_{N \to \infty} \frac{1}{b(N)|F_N|}\sum_{g \in F_N} a(g) \prod_{i=1}^r \tilde{f_i}(g_ig) \textrm{ for all }r \in \N, g_1,\dots,g_r \in G,\]
		and $\tilde{F_i} \in \{F_1,\dots,F_{\ell},\bar{F_1},\dots,\bar{F_{\ell}}\}$ and $\tilde{f_i} \in \{f_1,\dots,f_{\ell},\bar{f_1},\dots,\bar{f_{\ell}}\}$, respectively. Thus, $\mathds{Y}$ is a Furstenberg system for the tuple $(\{f_1,\dots,f_{\ell}\},(F_N),a,b)$, as desired. One can also check that $\nu$ satisfies $\textrm{supp}(\nu) \subseteq \overline{\{T_g \omega : g \in G\}}$.
	\end{proof}
	We are now in a position to establish the ``functional'' analog of Theorem \ref{mainthmproof}:
	\begin{theorem}
		Let $G$ be a countable amenable group and let $((F_N),a,b)$ be an averaging scheme. Let $\{f_1,\dots,f_{\ell}\}$ be a family of bounded functions on $G$ that is accordant with $((F_N),a,b)$, and let $\mathds{X}=(X,\mathcal{B}_X,\mu,(T_g)_{g \in G})$ be a Furstenberg system for the tuple $(\{f_1,\dots,f_{\ell}\},(F_N),a,b)$. Then, $\mathds{X}$ is isomorphic to the Furstenberg system $\mathds{Y}=(Y,\mathcal{B}_Y,\nu,(S_g)_{g \in G})$ constructed in Lemma \ref{sequencespace2}.
	\end{theorem}
	\begin{proof}
		The proof of this Theorem is very similar to the proof of Theorem \ref{mainthmproof}, and so we are only going to point out the major changes that have to be implemented. The first major step of the proof is done in the same way, but for the algebra \\ $\mathcal{A}=\overline{\textrm{Span}\{ \prod_{i=1}^r T_{g_i}\tilde{F_i},  r \in \N, g_1,\dots,g_r \in G, \tilde{F_i} \in \{F_1,\dots,F_{\ell}\} \}}^{L^{\infty}(\mu)}$. Arguing as in the proof of Theorem \ref{mainthmproof} one obtains an isomorphism $\mathds{X} \cong \mathds{Z}$, where $\mathds{Z}=(Z,\mathcal{B}_Z,\tilde{\mu},(\tilde{T_g})_{g \in G})$ is constructed in the same way as in the proof of Theorem \ref{mainthm} (but now for the algebra $\mathcal{A}$). 
		Let $Y=(\D^{\ell})^G$. 
		Define a map $\Phi$ from $Z$ to the compact space $Y$ by $\Phi(\chi)= ((\chi(T_gF_1),\dots,\chi(T_gF_{\ell}))_{g \in G}$. Note that $\Phi$ is well defined. Indeed, since $\chi \in Z$, and $\mathcal{A}$ is a unital $C^*$-algebra, we have that $\chi$ is a positive linear functional that respects conjugation. Now, $F_i(X) \subseteq \D$, whence $1-|F_i|^2$ is a positive element (of $L^{\infty}(\tilde{\mu})$ and $\chi(1)=1$, which implies that $|\chi(F_i)|^2 \leq 1$, so $\chi(F_i) \in \D$, as claimed. 
		Finally, arguing again as in the proof of Theorem \ref{mainthmproof} we can see that $\Phi$ is injective and continuous. Notice that, in particular, this means that $\Phi$ is measurable.
		Now, let $\nu=\Phi_{*}\mu$. The map $\Phi$ provides a measurable isomorphism between the measure preserving system $\mathds{Z}$ and the system $\mathds{Y}$ which was constructed in Lemma \ref{sequencespace2}. Letting $H_i=x_i(e)$, we see (again by Stone-Weierstrass) that the probability measure $\nu$ is determined by the values
		\[
		\int_Y \prod_{j=1}^r S_{g_j}\tilde{H_j} \ d\nu, \textrm{ for all } r \in \N, \ g_1,\dots,g_r \in G, \tilde{H_j} \in \{H_1,\dots,H_{\ell},\bar{H_1},\dots,\bar{H_{\ell}}\}.
		\]
		\begin{multline*}
		\textrm{We have: } \hspace{1cm}\int_Y \prod_{j=1}^r S_{g_j}\tilde{H_j} \ d\nu=\int_Z \prod_{j=1}^r S_{g_j}\tilde{H_j}(\Phi(\chi)) \ d\tilde{\mu}(\chi)= \\\int_Z \prod_{j=1}^r S_{g_j}\tilde{H_j}((\chi(S_gF_1),\dots,\chi(S_gF_{\ell}))_{g \in G}) \ d\tilde{\mu}(\chi)=\int_Z \chi\left(\prod_{i=1}^r T_{g_i}\tilde{F_i}\right) \ d\tilde{\mu}(\chi)=\\ \int_Z \Gamma\left(\prod_{i=1}^r T_{g_i}\tilde{F_i}\right) \ d\tilde{\mu}(\chi)=\int_X \prod_{i=1}^r T_{g_i}\tilde{F_i} \ d\mu= \lim_{N \to \infty} \frac{1}{b(N)|F_N|}\sum_{g \in F_N} a(g) \prod_{i=1}^r \tilde{f_i}(g_ig),
		\end{multline*}
		where we used the definition of the Gelfand isomorphism $\Gamma: \mathcal{A} \rightarrow C(Z)$ in the second to last step, and the fact that $\mathds{X}$ is a Furstenberg system for the tuple $(\{f_1,\dots,f_{\ell}\},(F_N),a,b)$ in the last. It follows that for $\omega=((f_1(g),\dots,f_{\ell}(g)))_{g \in G}$,
		\begin{equation}\label{weak2}
		\nu=\textrm{w*}\lim_{k \to \infty} \frac{1}{b(N)|F_N|}\sum_{g \in F_N} a(g)\delta_{S_g\omega},
		\end{equation}
		whence $\textrm{supp}(\nu) \subseteq \overline{\{ S_g\omega : g \in G\}}$. Notice that the right hand side in \eqref{weak2} is well defined due to the assumption that $\{f_1,\dots,f_{\ell}\}$ is accordant with the averaging scheme $((F_N),a,b)$. We are done.
	\end{proof}
	Let $((F_N),a,b)$ be an averaging scheme. As was done in Section 5, we can also go back from a system an ergodic system $(X,\mathcal{B},\mu,(T_g)_{g \in G})$ and bounded measurable functions $F_1,\dots,F_{\ell}: X \rightarrow \C$ to families of bounded functions $\{f_1,\dots,f_{\ell}\}$ that are accordant with $((F_N),a,b)$:
	\begin{theorem}\label{ergodicbackAppendix}
		Let $G$ be a countable amenable group.  Let $(X,\mathcal{B},\mu,(T_g)_{g \in G})$ be an ergodic measure preserving system. Let $F_1,\dots,F_{\ell} \in L^{\infty}(\mu)$ and let $((F_N),a,b)$ be an averaging scheme. Then there exists a subsequence $(F_{N_k})$ and a family of $\C$-valued bounded functions $\{f_1,\dots,f_{\ell}\}$ such that
		\begin{equation}\label{mirrorfamily} \int_X \prod_{i=1}^r T_{g_i}\tilde{F_i} \ d\mu=\lim_{k \to \infty} \frac{1}{b(N_k)|F_{N_k}|}\sum_{g \in F_{N_k}} a(g)\prod_{i=1}^r \tilde{f_i}(g_ig),\end{equation}
		for all $r \in \N$ and $g_1,\dots,g_r \in G$, where $\tilde{F_i} \in \{F_1,\dots,F_{\ell},\bar{F_1},\dots,\bar{F_{\ell}}\}$ and $\tilde{f_i} \in \{f_1,\dots,f_{\ell},\bar{f_1},\dots,\bar{f_{\ell}}\}$ respectively. 
	\end{theorem}
	\begin{proof} We argue as in the proof of Theorem \ref{ergodicback} for $\mathcal{F}:=\{\prod_{i=1}^m T_{g_i}\tilde{F_i} : r \in \N, g_1,\dots,g_r \in G, \tilde{F_i} \in \{F_1,\dots,F_{\ell},\bar{F_1},\dots,\bar{F_{\ell}}\}\}$ to obtain a full measure subset $X_0$ such that for all $x \in X_0$ and $H \in \mathcal{F}$ we have	
		\begin{equation}
		\lim_{k \to \infty} \frac{1}{b(N_k)|F_{N_k}|}\sum_{g \in F_{N_k}} a(g)H(T_gx)=\int_X H \ d\mu.
		\end{equation}
		Now, for each $x \in X_0$ we can let $f_{i,x}(g)=F_i(T_gx)$, $i=1,\dots,\ell$, and one easily checks that for all $x \in X_0$ the family $\{f_{1,x},\dots,f_{\ell,x}\}$ satisfies \eqref{mirrorfamily} as desired.
	\end{proof}
	
\end{document}